\documentclass[11pt]{amsart}
\usepackage[utf8]{inputenc}

\usepackage{bm}
\usepackage{amssymb}
\usepackage[all]{xy}
\usepackage{tikz-cd}
\usepackage{bbm}
\usepackage{hyperref}

\theoremstyle{plain}

\numberwithin{equation}{section}

\newtheorem{theorem}{Theorem}[section]
\newtheorem{lemma}[theorem]{Lemma}

\newtheorem{corollary}[theorem]{Corollary}
\newtheorem{proposition}[theorem]{Proposition}

\newcommand{\C}{\mathbb{C}}

\theoremstyle{definition}
\newtheorem{definition}[theorem]{Definition}

\theoremstyle{remark}

\newcommand{\cC}{{\mathcal C}}
\newcommand{\Rep}{\operatorname{Rep}}
\newcommand{\cD}{{\mathcal D}}

\newcommand{\ot}{{\otimes}}
\newcommand{\ku}{\mathbbm{k}}

\newcommand{\Z}{\mathbb{Z}}
\newcounter{commentcounter}

\title{Equivariant fusion subcategories}

\author[C\'esar Galindo]{C\'esar Galindo}
\address{ Departamento de Matem\'aticas, Universidad de los Andes, Bogot\'a, Colombia}
\email{cn.galindo1116@uniandes.edu.co}

\author[Corey Jones]{Corey Jones}
\address{ Department of Mathematics, North Carolina State University }
\email{cmjones6@ncsu.edu}

\date{}

\begin{document}

\begin{abstract}
We provide a parameterization of all fusion subcategories of the equivariantization by a group action on a fusion category. As applications, we classify the Hopf subalgebras of a family of semisimple Hopf algebras of Kac-Paljutkin type and recover Naidu-Nikshych-Witherspoon classification of the fusion subcategories of the representation category of a twisted quantum double of a finite group. 
\end{abstract}

\thanks{C\'esar Galindo was partially supported by Faculty of Science of Universidad de los Andes, Convocatoria
2020-2022 para la Financiación de Programas de Investigación. Corey Jones was supported by NSF grant DMS 1901082/2100531.
 }
\maketitle
\section{Introduction}

Fusion categories generalize the representation categories of finite-di\-men\-sional semi-simple (quasi) Hopf algebras \cite{ENO}.  
A fundamental construction in the theory of fusion categories is equivariantization by a group action, \cite{EGNO}. Given a categorical action of a finite group $G$ on a fusion category $\mathcal{C}$, the equivariantization $\cC^{G}$ is a new fusion category, consisting of categorical ``fixed points" under the $G$ action. Equivariantization, and its reciprocal construction called de-equivariantization, have been extensively studied and used in recent years to classify and characterize families of fusion categories, see \cite{ENO2}, \cite{DGNO}. The goal of this note is to identify the lattice of fusion subcategories of $\cC^{G}$ in terms of the lattice of fusion subcategory of $\cC$ and cohomological data arising from the $G$-action (Theorem \ref{MainTheorem}).

If a fusion category $\cC$ is equivalent to the category of (right) comodules of a finite-dimensional semisimple Hopf algebra $H$, then the lattice of fusion subcategories of $\cC$ is isomorphic to the lattice of Hopf subalgebras of $H$, see \cite[Theorem 6]{MR1375826}. Many interesting Hopf algebras (in particular  Hopf algebras of Kac-Paljutkin type and quantum groups at roots of unity) contain a central Hopf subalgebra $\C^G$, and this implies the existence of a central exact sequence of Hopf algebras of the form \[\ku \to \ku^G\to H\to Q\to \ku.\] In this case, the category of $H$-comodules is a $G$-equivariantization of the category of $Q$-comodules, \cite{MR3160718}.  Thus we can apply our results to obtain a parameterization of the Hopf subalgebras of central extensions of semisimple Hopf algebras, particularly to  Hopf algebras of Kac-Paljutkin type associated to semidirect products of groups.

It is well known how to parameterize the simple objects of the equivariantization and compute their fusion rules in terms of the data of the $G$ action (for example, \cite{MR3059899}). Fusion subcategories of a fusion category $\cD$ correspond with fusion sub-rules of the Grothendieck ring $K_0(\cD)$. Hence, in principle, having the fusion rules of $\cC^G$, the computation of fusion subcategories of $\cC^G$ reduces to the computation of fusion sub-rules of $K_0(\cC^G)$. However, even the computation of $K_0(\cC^G)$ and its multiplication could be computationally demanding. Therefore, it is desirable to obtain a parameterization of the fusion subcategories of an equivariantization directly in terms of the data of the group action.

After posting an initial version of this paper to the arXiv, Dmitri Nikshych brought to our attention that our Theorem \ref{MainTheorem} was previously obtained by Alex Levin and presented at conferences in 2017. Levin's work has not yet been made publicly available. We thank Dmitri for bringing this to our attention.

The paper is organized as follows. In Sections \ref{preliminares} and \ref{SeqEq} we recall the basic definitions and results concerning equivariantization by a categorical group action and de-equivariantization of a fusion containing a central inclusion of $\operatorname{Rep}(G)$. In Section \ref{main results}, we prove our main result, which provides a parameterization of equivariant fusion subcategories of a fusion category. In Section \ref{pointed examples}, we take a closer look at the pa\-ra\-me\-te\-rization in the case of group actions on pointed fusion categories, providing an obstruction theory for $G$-invariant trivializations. Finally, in Section \ref{application}, we apply this to parameterize Hopf subalgebras of a class of Kac-Paljutkin type studied in \cite{MR1307564} and \cite{MR1448809}.

\section{Preliminaries}\label{preliminares}

A fusion category is a finitely semi-simple rigid tensor category. In this paper, we assume our field is algebraically closed and characteristic $0$. We refer the reader to the comprehensive reference \cite{EGNO} for categorical background. In this section, we will recall some basic notions and set up notation for categorical actions of groups on fusion categories and equivariantization.

Let $\underline{\text{Aut}}(\mathcal{C})$ denote the 2-group of tensor autoequivalences of a fusion category and tensor natural isomorphisms with tensor product given by composition of tensor functors. A \textit{categorical action} of $G$ on $\mathcal{C}$ is a strong monoidal functor $ \underline{G}\rightarrow \underline{\text{Aut}}(\mathcal{C}) $. We denote the assignment $g\mapsto g_{*}$, and the monoidal natural isomorphisms $\mu_{g,h}: g_{*}\circ h_{*}\rightarrow gh_{*}$. We denote the component of a natural isomorphism with a superscript, i.e., the $x$ component of $\mu_{g,h}$ is denoted $\mu^{x}_{g,h}$ for $x\in \cC$.

Given a categorical action of $G$ on $\mathcal{C}$, recall that the equivariantization $\mathcal{C}^{G}$ is a new fusion category whose objects are \textit{equivariant objects}, given by pairs $(x, \rho)$, where $x\in \mathcal{C}$ and $\rho=\{\rho_{g}: g_{*}(x)\cong x\}_{g\in G}\ $ is a family of isomorphisms satisfying $\mu^{x}_{g,h}\rho_{g}g_{*}(\rho_{h})=\rho_{gh}$ for all $g,h  \in G$. Morphisms between equivariant objects $(x, \rho)$ and $(y, \delta)$ are morphisms $f\in \cC(x,y)$ such that for all $g\in G$, $f\rho_{g}=\delta_{g}g_{*}(f)$.

This forms a new monoidal category called the equivariantization, denoted $\mathcal{C}^{G}$. The product $(x,\rho)\otimes (y, \delta)$ is defined by

$$(x\otimes y, (\rho\otimes \delta)\circ\beta^{x,y}),$$
where $\beta^{x,y}_{g}:g_{*}(x\otimes y)\cong g_{*}(x)\otimes g_{*}(y)$ is the tensorator of the monoidal equivalence $g_{*}$.

An important point of note is that since the unit object of $\mathcal{C}$ is simple, there is a canonical copy of $\text{Rep}(G)\le \mathcal{C}^{G}$ embedded as a full subcategory, consisting of equivariant objects which are direct sums of the unit object $\mathbbm{1}$. Furthermore, this subcategory naturally lifts to the center $\mathcal{Z}(\mathcal{C}^{G})$. Fusion categories $\cC$ equipped with a fully-faithful monoidal functor $F:\text{Rep}(G)\rightarrow \mathcal{C}$ that factors through a braided monoidal functor $\text{Rep}(G)\rightarrow \mathcal{Z}(\mathcal{C})$ fusion categories \textit{over} $\text{Rep}(G)$.

Both $G$-fusion categories and fusion categories over $\text{Rep}(G)$ naturally form 2-categories (see \cite[Section 4]{DGNO}). Equivariantization extends to a 2-functor from $G$-fusion categories to fusion categories over $\text{Rep}(G)$. This 2-functor has an explicit inverse called the \textit{de-equivariantization functor}.

Given a fusion category over $\text{Rep}(G)$, then identifying $\text{Rep}(G)$ with its image under $F$, the etale algebra $\text{Fun}(G)\in \text{Rep}(G)$ can be viewed as a central algebra in $\mathcal{C}$. Thus the quotient category of right $\text{Fun}(G)$ modules $\mathcal{C}_{\text{Fun}(G)}$ is a fusion category called the \textit{de-equivariantization} of $\mathcal{C}$, which we also denote by $\mathcal{C}_{G}$. The free module functor $\mathcal{C}\rightarrow \mathcal{C}_{\text{Fun}(G)}=\cC_{G}$ is monoidal. Furthermore, one can equip the dequivariantization with a natural categorical action of $G$, and in fact this extends to a 2-functor from the 2-category of fusion categories over $\text{Rep}(G)$ to the 2-category of $G$-fusion categories. The fundamental theorem of $G$-actions on fusion categories states that equivariantization and de-equivariantization are mutually inverse 2-functors \cite[Theorem 4.4]{DGNO}.

\section{Sequential Equivariantization}\label{SeqEq}

Let $G$ be a group acting on a monoidal category $\cC$ and let $H\le G$ be a normal subgroup. First pick a section $\iota:G/H\rightarrow G$ of the canonical projection $\pi: G\rightarrow G/H$ and  define $\alpha: G/H \times G/H \rightarrow H$ by $\alpha(a,b):=\iota(a)\iota(b)\iota(ab)^{-1}$ for all $a,b\in G$.

Using the bijection
\begin{align*}
    H\times G/H \to  G\\
    (h,a)\mapsto h\iota(a).
\end{align*}
the multiplication on $G$ transports to \[(h,a)(k,b)=(h\iota(a)k\iota(a)^{-1}\alpha(a,b), ab).\]

Given $(x,\rho)\in \cC^{H}$, and $a\in G/H$, we define 
$a_{*}(x,\rho):=(x^{\prime},\rho^{\prime}) $ where
\[x^{\prime}:=\iota(a)_{*}(x)\]
and  $\rho^{\prime}_{h}$ is defined via the composition

\[\begin{tikzcd}
   h_{*}(\iota(a)_{*}(x))\arrow{r}{\textbf{can}} & \iota(a)(\iota(a)^{-1}h\iota(a)_{*}(x))\arrow{r} {\iota(a)_{*}(\rho_{h})} & \iota(a)_{*}(x).
 \end{tikzcd}\]
 
 It is straightforward to check for each $a\in G/H$ this is a monoidal functor, using the tensorator for $\iota(a)_{*}$.
 
 To obtain a categorical action, we need monoidal natural isomorphisms $a_{*}\circ b_{*}\rightarrow (ab)_{*}$. Recall that $\iota(a)\iota(b)=\iota(ab)(\iota(ab)^{-1}\alpha(a,b)\iota(ab))$, and thus we define $\nu^{(x,\rho)}_{a,b}$ to be the morphism  making the following diagram commute
 
 \[\begin{tikzcd}
  \iota(a)_{*}(\iota(b)_{*}(x))\arrow{r}{\textbf{can}}\arrow[swap]{dr}{\nu^{(x,\rho)}_{a,b}}& \iota(ab)_{*}((\iota(ab)^{-1}\alpha(a,b)\iota(ab))_{*}(x))\arrow{d}{\iota(ab)_{*}(\rho_{\iota(ab)^{-1}\iota(a)\iota(b)})}\\
   & \iota(ab)_{*}(x)
 \end{tikzcd}\]
 
 \noindent which is an equivariant morphism. $\nu_{a,b}: a_{*}\circ b_{*}\rightarrow ab_{*}$ assembles into a monoidal natural isomorphism. This gives us a categorical action $\underline{G/H}\rightarrow \underline{\text{Aut}}(\cC^{H})$.
 
It is not difficult to see there is a canonical monoidal equivalence $(\mathcal{C}^{H})^{G/H}\cong \mathcal{C}^{G}$ (for example \cite[Equation 75]{DGNO}). One way to see this is using the strictification result of \cite{Galindo-coherencia}, which implies it suffices to consider the case where the $G$ action on $\mathcal{C}$ is strict. Here it is very easy to build this equivalence by simply rearranging the data of equivariant objects.

%%%%%%%%%%%%%%%%%%%%%%%%%%

\section{The Parametrization}\label{main results}

In this section we parameterize \textit{fusion subcategories} of an equivariantization of a fusion category by a group action. If $\cC$ is a fusion category, by a fusion subcategory we mean a full, replete monoidal subcategory (which itself is necessarily a fusion category). These correspond bijectively to unital based subrings of the fusion ring of $\cC$.

Recall that the trivial categorical action $\text{Tr}$ of a group $G$ on a monoidal category $\cC$ sends each group element to the identity monoidal functor $\text{Id}_{\cC}$, and all structure maps are identities as well.  If we have a categorical action $\alpha:\underline{G}\rightarrow \underline{\text{Aut}}(\cC)$, a \textit{trivialization} is a monoidal natural isomorphism $\alpha\cong \text{Tr}$. Unpacking this, a trivialization consists of 

\begin{itemize}
    \item 
    For each $g\in G$, a monoidal natural isomorphism $\eta_{g}: g_{*}\cong \text{Id}_{\cC}$.
    \item
    These $\{\eta_{g}\}_{g\in G}$ must satisfy the equation
\begin{equation}\label{eq: trivialization equ}
\eta^{x}_{g}\alpha_{g}(\eta^{x}_{h})= \eta^{x}_{gh}\mu^{x}_{g,h}.
\end{equation}

\end{itemize}

Recall if $\mathcal{E}$ is a full, replete subcategory of $\cC$ we say $\mathcal{E}$ is \textit{invariant} under the $G$-action if  $\alpha_{g}$ restricts to an autoequivalence of $\mathcal{E}$.

\begin{definition}
Let $\mathcal{E}\le \cC$ be  a $G$-invariant fusion subcategory, and let $H\le G$ be a normal subgroup. We say a trivialization $\eta$ of $ \alpha|_{H}$ on the category $\mathcal{E}$ is $G$-equivariant if the following diagram commutes:

\begin{equation}\label{diag: condition equivariant lifting}
\begin{tikzcd}
   h_{*}(g_{*}(x))\arrow{r}{\textbf{can}}\arrow[swap]{dr}{\eta^{g_{*}(x)}_{h}} & g_{*}(g^{-1}hg_{*}(x))\arrow{d}{g_{*}(\eta^{x}_{g^{-1}hg}(x))} \\
     & g_{*}(x)
 \end{tikzcd}    
\end{equation}

where $\textbf{can}$ is the canonical morphism arising from coherence.

\end{definition}

Writing $\textbf{can}$ out explicitly gives the equation:

$$g(\eta^{x}_{g^-1hg}\mu^{x}_{g^{-1},hg}\mu^{x}_{h,g})(\mu^{h(g(x))}_{g,g^{-1}})^{-1}=\eta^{g_{*}(x)}_{h}. $$

%%%%%%%%%%%%%%%%%%%%%%%%%

\begin{definition}

Given a group $G$ acting on a monoidal category $\mathcal{C}$, a \textit{lifting} of a fusion subcategory $\mathcal{E}\le \mathcal{C}$ is a choice, for every $x\in \mathcal{E}$ of equivariant structure $(x, \rho^{x})$ such that

\begin{enumerate}
\item \label{item: lifting1}
$(x,\rho^{x})\otimes_{\mathcal{C}^{G}} (y,\rho^{y})=(x\otimes y, \rho^{x\otimes y})$

\item \label{item: lifting 2}
$\text{Hom}_{\mathcal{C}^{G}}((x,\rho^{x}),(y,\rho^{y}))=\text{Hom}_{\mathcal{C}}(x,y)$.

\end{enumerate}

We denote the image of $\mathcal{E}$ in $\mathcal{C}^{G}$ by $\mathcal{E}^{\rho}\le \mathcal{C}^{G}$.

\end{definition}

\begin{lemma}\label{lemma: lifting=trivialization}
The data of a lifting is precisely the data of a trivialization of the $G$-action on $\mathcal{E}$.
\end{lemma}

\begin{proof}

Given a lifting of $\mathcal{E}$ we can define a trivialization of the $G$ action on $\mathcal{E}$ as follows. Set $\eta^{x}_{g}: g(x)\cong x$ by $\eta^{x}_{g}:=\rho^{x}_{g}$. Equation \eqref{eq: trivialization equ} is  \eqref{item: lifting1} in the definition of lifting and \eqref{item: lifting 2} is equivalent to the naturality of $\eta_g$ for every $g$.

Conversely, given a trivialization of $G$ on $\mathcal{E}$, then $(x, \eta^{x}_{g})$ defines an equivariant structure. Since the $\eta_{g}$ are natural, we see $$\text{Hom}_{\mathcal{C}^{G}}((x,\rho^{x}),(y,\rho^{y}))=\text{Hom}_{\mathcal{C}}(x,y).$$

\end{proof}

\begin{lemma}\label{lemma: equivariant liftings} Given a $G$ action on $\mathcal{C}$, $\mathcal{E}\le \mathcal{C}$ be a fusion subcategory, and a lifting $\rho$ of $\mathcal{E}$ to $\mathcal{C}^{H}$ for some normal subgroup $H\le G$, then the corresponding trivialization of the $H$ action on $\mathcal{E}$ is $G$-equivariant if and only if $\mathcal{E}^{\rho}\le \mathcal{C}^{H}$ is invariant under the $G$ action.
\end{lemma}

\begin{proof}
Let $\rho$ be a lifting of $\mathcal{E}$ to $\cC^H$, and let $\{\eta_{h}\}_{h\in H}$ be the trivialization defined by $\eta^{x}_{h}:=\rho^{x}_{h}$. If $g\in G$  then $g_*(x,\rho)=(x',\rho')$ where $x'=g_*(x)$ and 

\[\begin{tikzcd}
   h_{*}(g_{*}(x))\arrow{r}{\textbf{can}} & g(g^{-1}hg_{*}(x))\arrow{r} {g_{*}(\rho_{h})} & g_{*}(x).
 \end{tikzcd}\]
 
Hence $\mathcal{E}^\rho\le \cC^H$ is $G$-invariant if and only if $\rho'=\rho_{g_*(x)}$, and that is exactly the commutativity of diagram \eqref{diag: condition equivariant lifting}. 
\end{proof}

We record the following, defined in \cite{MR3943750}:

\begin{definition}
If $\mathcal{C}$ is a fusion category, $A\in \cC$ is a connected separable algebra, and $\mathcal{E}\le \mathcal{C}$ is fusion subcategory, we say $\mathcal{E}$ is \textit{transversal} to $A$ if  $A \cap \mathcal{E}=\mathbbm {1}$. 
\end{definition}

It's easy to see that transversality is equivalent to requiring $$\text{Hom}_{\mathcal{C}}(x, \mathbbm{A})\cong\text{Hom}_{\mathcal{C}}(x,\mathbbm{1})$$ for all $x\in \mathcal{E}$. We denote by $\mathcal{C}_{A}$ the semi-simple category of right $A$ modules in $\mathcal{C}$. We have the following lemma:

\begin{lemma}\label{transversallemma} Let $\mathcal{C}$ be fusion, $\mathcal{E}\le \mathcal{C}$ a fusion subcategory, and $A\in \mathcal{C}$ a connected seperable algebra. Then $\mathcal{E}$ is transversal to $A$ if and only if the free module functor $F_{A}: \mathcal{E}\rightarrow \mathcal{C}_{A}$, $F_{A}(x):=x\otimes A$, is fully faithful.
\end{lemma}

\begin{proof}
Let $x,y\in \mathcal{E}$. Suppose $\mathcal{E}$ is transversal to $A$. Recall that the left adjoint of $F_{A}$ is simply the forgetful functor, i.e. the functor that forgets the right $A$-module structure on an object. Then $\text{Hom}_{\mathcal{C}_{A}}(F_{A}(x), F_{A}(y))\cong \text{Hom}_{\mathcal{C}}(x, F_{A}(y))\cong \text{Hom}_{\cC}(y^{*}\otimes x, A)\cong \text{Hom}_{\cC}(y^{*}\otimes x, \mathbbm{1})\cong \text{Hom}_{\cC}( x, y)$. Thus the inclusion $$F_{A}: \text{Hom}_{\cC}( x, y)\rightarrow \text{Hom}_{\mathcal{C}_{A}}(F_{A}(x), F_{A}(y)) $$ is an isomorphism, hence $F_{A}$ is fully faithful.

Conversely, if $F_{A}$ is fully faithful on $\mathcal{E}$, then $$\text{Hom}_{\mathcal{C}}(x, \mathbbm{1})\cong \text{Hom}_{\mathcal{C}_{A}}(F_{A}(x), F_{A}(\mathbbm{1}))
\cong \text{Hom}_{\mathcal{C}}(x, A),$$ so $\mathcal{E}$ is transversal to $A$.
\end{proof}

\begin{lemma}\label{lemma: equivalentsubcategorycharacterization}
Let $\cC$ be a fusion category with $G$ action. There is a bijection between the following:

\begin{enumerate}
    \item     Fusion subcategories $\mathcal{D}\le \cC$ with trivializations of the $G$ action on $\mathcal{D}$.
    \item
    Fusion subcategories $\mathcal{E}\le \cC^{G}$ such that $F_{G}|_{\mathcal{E}}$ is fully faithful.
    \item
    Fusion subcategories of $\cC^{G}$ which are transversal to the algebra $\text{Fun}(G)\in \operatorname{Rep}(G)\le \cC^{G}$.
\end{enumerate}
\end{lemma}

\begin{proof}

First we establish a bijection between items $(1)$ and $(2)$. Given $\mathcal{D}\le \cC$ and a trivialization $\eta$, this uniquely defines a lifting $\rho$ by Lemma \ref{lemma: lifting=trivialization}. Then we consider the full subcategory $\mathcal{E}:=\mathcal{D}^{\rho}\le \mathcal{C}$, spanned by objects $\{(x, \rho^{x})\}_{x\in \mathcal{D}}$. By the definition of lifting, this is a full monodial catgeory of $\cC^{G}$, such that the forgetful functor is fully faithful. 

In the other direction, suppose we have a $\mathcal{E}\le \cC^{G}$ such that the forgetful functor is fully faithful. Define $\mathcal{D}$ to be the image of $\mathcal{E}$ under the forgetful functor. Since the forgetful functor is fully faithful we have $\text{Hom}_{\mathcal{E}}((x,\eta^x),(y,\eta^y))=\text{Hom}_\cC(x,y)$, in particular, that implies that if $(x,\eta^x), (x,\lambda^x)\in \mathcal{E}$, then $\eta^x=\lambda^x$. Since $\mathcal{E}$ is monoidal, the uniqueness of $\eta^x$ for all $x\in \cD$ implies  \eqref{eq: trivialization equ}. Now for each object $x\in D$, by definition (and repleteness) there exists a unique pre-image $(x, \eta^{x})\in \mathcal{E}$. The isomorphisms $\{\eta^x\}_{x\in \mathcal{D}}$ define a lifting of $\mathcal{D}$.

The equivalence of items $(2)$ and $(3)$ follows from Lemma \ref{transversallemma}.

\end{proof}

\begin{theorem}\label{MainTheorem}
Let $\cC$ be a fusion category with $G$ action. Fusion subcategories of $\mathcal{C}^{G}$ are parameterized by triples $(\mathcal{E}, H, \eta)$ where

\begin{enumerate}
    \item 
    $\mathcal{E}\le \mathcal{C}$ is a $G$-invariant fusion subcategory.
    \item
    $H\le G$ is a normal subgroup.
    \item
    $\eta$ is a $G$-equivariant trivialization of $H$ on $\mathcal{E}$.
\end{enumerate}
    
\end{theorem}

\begin{proof}
Given the above data, let $\rho$ be the lift of $\mathcal{E}$ to $\mathcal{C}^{H}$ defined by the trivialization (Lemma \ref{lemma: lifting=trivialization}). The fusion subcategory $\mathcal{E}^{\rho}\le \mathcal{C}^{H}$ is $G/H$ invariant by Lemma \ref{lemma: equivariant liftings}. Define $\mathcal{D}$ to be the replete image of $(\mathcal{E}^{\rho})^{G/H}$ under the canonical equivalence $(\cC^{H})^{G/H}\cong \cC^{G}$. This gives a subcategory of $\cC^{G}$ from our data.

In the other direction, let $\mathcal{D}\le \mathcal{C}^{G}$ be a fusion subcategory. 

Let $\mathcal{E}\le \cC$ be the full fusion subcategory generated by the image of $\mathcal{D}$ under the forgetful functor $F_{G}: \mathcal{C}^{G}\rightarrow \mathcal{C}$. Clearly $\mathcal{E}$ is $G$-invariant.

Let $H$ be defined by $\text{Rep}(G)\cap \mathcal{D}\cong \text{Rep}(G/H)$. To get the trivialization of the $H$ action on $\mathcal{E}$, $\eta$, 

Note that we have the following diagram of monoidal functors commutes up to canonical natural isomorphism:

\begin{equation}
\begin{tikzcd}
   \cC^{G}\arrow{dr}{F_{G/H}}\arrow[swap]{dd}{F_{G}} &  \\
     & \cC^{H}\arrow{dl}{F_{H}}\\
    \cC & \\
 \end{tikzcd}    
\end{equation}

Thus if we define $\mathcal{E}^{\prime}=F_{G/H}(\mathcal{D})\le \cC^{H}$, then we have $F_{H}(\mathcal{E}^{\prime})=F_{G}(\mathcal{D})=\mathcal{E}$.
Alternatively, using the equivalence $(\mathcal{C}^{G})_{G/H}\cong \mathcal{C}^{H}$, then we see  $\mathcal{E}^{\prime}:=(\cD)_{G/H}\le \cC^{H}$.

We claim that the separable, connected algebra $\text{Fun}(H)\in \cC_{G/H}$ is transversal to $\mathcal{E}^{\prime}$, or in other words, $\text{Fun}(H)\cap \mathcal{E}^{\prime}\cong \mathbbm{1}$. Recall there is a correspondence between etale algebras in $\mathcal{E}^{\prime}\cap F_{G/H}(\Rep(G))\le F_{G/H}(\Rep(G))\cong \Rep(H) $ and etale algebras over $\text{Fun}(G/H)$ in $\Rep(G)\cap \mathcal{D}\cong \text{Rep}(G/H)$. But since $\text{Fun}(G/H)$ is maximal etale in $\Rep(G/H)$ and corresponds to $\mathbbm{1}$ in $\mathcal{E}^{\prime}\cap F_{G/H}(\Rep(G))$  this proves $\text{Fun}(H)$ is transversal to $\mathcal{E}^{\prime}$.

By Lemma \ref{lemma: equivalentsubcategorycharacterization}, this implies $F_{H}$ induces an equivalence from $\mathcal{E}^{\prime}$ to $\mathcal{E}$, and such an equivalence uniquely defines a lifting $\mathcal{E}\rightarrow \mathcal{C}^{H}$ which is $G/H$ invariant (since $\mathcal{E}^{\prime}$ is $G/H$ invariant by construction). This gives a $G$-equivariant trivialization $\eta$ by Lemma \ref{lemma: lifting=trivialization}.

Now, it is easy to see these two constructions are mutually inverse.

\end{proof}

From the construction, we have the following corollary

\begin{corollary}
Let $\mathcal{D}(\mathcal{E}, H, \eta)\le \cC^{G}$ denote the subcategory constructed in the previous theorem. Then

\begin{enumerate}
    \item 
    $\text{FPdim}(\mathcal{D}(\mathcal{E}, H, \eta))=[G:H] \text{FPdim}(\mathcal{E})$.
    \item
    $\mathcal{D}(\mathcal{E}, H, \eta)\le \mathcal{D}(\mathcal{F}, K, \nu)$ if and only if $\mathcal{E}\le \mathcal{F}$, $K\le H$ and $\eta|_{K}=\nu$ on $\mathcal{E}$.
\end{enumerate}
\end{corollary}

\section{Pointed fusion categories}\label{pointed examples}

Let $K$ be a finite group and $\omega\in Z^{3}(K, \mathbbm{k}^{\times})$. The data of a categorical action of the group $G$ on $\text{Vec}(K, \omega)$ can be described by the following:

\begin{enumerate}
    \item 
    A homomorphism $G\rightarrow \text{Aut}(K)$, $g\mapsto g_{*}$
    \item
    A collection of scalars $\beta^{k,l}_{g}\in \mathbbm{k}^{\times}$ for $k,l\in K, g\in G$ and a collection of scalars $\mu^{k}_{g,h}\in \mathbbm{k}^{\times}$ for $k\in K$ and $g,h\in G$, satisying the following equations:
    
\begin{equation}\label{eq: equation beta}
\frac{\omega_{k,l,m}}{\omega_{g_{*}(k),g_{*}(l),g_{*}(m)}}=\frac{\beta^{k,l}_{g}\beta^{kl,m}_{g}}{\beta^{l,m}_{g}\beta^{k,lm}_{g}}
\end{equation}

\begin{equation}\label{eq: 2-cocycle condition mu}
\mu^{k}_{g,h}\mu^{k}_{f,gh}=\mu^{h(k)}_{f,g}\mu^{k}_{fg,h}
\end{equation}

\begin{equation}\label{eq: compatibility mu and beta}
\frac{\beta^{k,l}_{gh}}{\beta^{h_{*}(k),h_{*}(l)}_{g}\beta^{k,l}_{h}}=\frac{\mu^{k}_{g,h}\mu^{l}_{g,h}}{\mu^{kl}_{g,h}}
\end{equation}

\end{enumerate}
To realize this data, pick a skeletal model of $\text{Vec}(K, \omega)$ with the associator given by $\omega$ explicitly. Then extend the action of $g_{*}$ on $K$ linearly to a functor. By semi-simplicity it suffices to define natural transformations of functors on simple objects. The tensorator $g_{*}(k\otimes l)\cong g_{*}(k)\otimes g_{*}(l)$ is given by $\beta^{k,l}_{g}$, and equation \eqref{eq: equation beta} above is exactly the required compatibility with the associator. We define the natural isomorphisms $gh_{*}(k)\cong g_{*}\circ h_{*}(k) $ by $\mu^{k}_{g,h}$. Equation \eqref{eq: compatibility mu and beta} ensures these are monoidal natural isomorphisms, and equation \eqref{eq: 2-cocycle condition mu} guarantees these assemble into a categorical action.

Now, our goal is to apply our general results to parameterize subcategories of the equivariantization in terms of the above data. The only remaining task is to interpret the G-equivariant trivializations in terms of this data.

Let $L\le K$ be a G-invariant subgroup, and $H\le G$ a normal subgroup such that $h_{*}|_{L}$ is trivial. First we can unpack the definition of a trivialization of the H action on $L$. 

It consists of a family $\eta^{k}_{h}\in \mathbbm{k}^{\times}$ for $k\in L, h\in H$ satisfying for all $g,h\in H$

\begin{equation}\label{eq: eta is coboundary of beta}
\eta^{lk}_{g}=\eta^{l}_{g}\eta^{k}_{g}\beta^{l,k}_{g}
\end{equation}

\begin{equation}\label{eq: eta is coboundary of mu}
\eta^{l}_{g}\eta^{l}_{h}=\eta^{l}_{gh}\mu^{l}_{g,h}
\end{equation}

Viewing $\eta: L\times H\rightarrow \mathbbm{k}$, we see that it is almost a bicharacter. In particular, it is a character in each variable up to $\beta$ and $\mu$ respectively. Thus we will call $\eta: L\times H\rightarrow \mathbbm{k}$ satisfying the above equations a $(\beta, \mu)$-bicharacter. 

Now, we can unpack the definition of a $G$-equivariant tivialization to obtain

\begin{equation}\eta^{g_{*}(k)}_{h}=\eta^{k}_{g^{-1}hg}\left( \frac{\mu^{k}_{g^{-1},hg} \mu^{k}_{h,g}}{\mu^{gh_{*}(k)}_{g,g^{-1}}}\right)
\end{equation}

or equivalently 

\begin{equation}\label{bicharacter}\eta^{g_{*}(k)}_{ghg^{-1}}=\eta^{k}_{h}\left( \frac{\mu^{k}_{g^{-1},gh} \mu^{k}_{ghg^{-1},g}}{\mu^{g(k)}_{g,g^{-1}}}\right)
\end{equation}

\noindent where we have used in the superscript on the denominator that $H$ is normal in $G$, and since $H$ acts trivially on $L$, $g(ghg^{-1})_{*}(k)=g(k)$ for all $k\in L$.

Note that we can assume $\mu^{k}_{h,g}$ is \textit{normalized} so that $\mu^{k}_{h,g}=1$ for all $k$ if $h$ or $g$ is 1. Then the defining condition for $\mu$ implies $\mu^{k}_{g^{-1},g}=\mu^{g_{*}(k)}_{g,g^{-1}}$. Furthermore, the defining equation, together with normalization and the fact that $h$ acts trivially on $L$ implies equation 

$$\mu^{k}_{g,h}\mu^{k}_{g^{-1},gh}=\mu^{k}_{g^{-1},g}.$$

\noindent Solving for $\mu^{k}_{g^{-1},gh}$, replacing $\mu^{k}_{g^{-1},g}$ with $\mu^{g_{*}(k)}_{g,g^{-1}}$, and substituting into equation \eqref{bicharacter} gives

\begin{equation}\label{bicharacter-2}\eta^{g_{*}(k)}_{ghg^{-1}}=\eta^{k}_{h}\left( \frac{ \mu^{k}_{ghg^{-1},g}}{\mu^{k}_{g,h}}\right).
\end{equation}

\begin{definition} A $(\beta,\mu)$ bicharacter satisfying equations \eqref{bicharacter} or \eqref{bicharacter-2} is called a \textit{G-equivariant} $(\beta, \mu)$ bicharacter.
\end{definition}

\begin{corollary}\label{subcategoriespointed}
Consider a categorical action of $G$ on $\text{Vec}(K,\omega)$, given by a homomorphism $G\rightarrow \text{Aut}(K)$, together with $(\beta,\mu)$ as above. Then fusion subcategories of  $\text{Vec}(K, \omega)^{G}$ are parameterized by triples $(L,H,\eta)$ where $L\le K$ is a G-invariant subgroup, $H$ is a normal subgroup of $G$ fixing every element of $L$, and $\eta$ is a $G$-invariant $(\beta, \mu)$ bicharacter.
\end{corollary}

\subsection{$\mathcal{Z}(\text{Vec}(G,\omega))$}

Now we compare our parametrizations with the results of \cite{MR2552301}. To state their results, we need to introduce some notation. Let $\omega\in Z^{3}(G, \mathbbm{k}^{\times})$. Following \cite[Section 5.1]{MR2552301}), for any $a,x,y\in G$ define the quantities

$$\eta_{a}(x,y):=\frac{\omega(x,y,a)\omega(xyay^{-1}x^{-1},x,y)}{\omega(x,yay^{-1},y)}$$

$$\nu_{a}(x,y):=\frac{\omega(axa^{-1},aya^{-1},a)\omega(a,x,y)}{\omega(axa^{-1},a,y)}$$

Setting $\mu^{k}_{g,h}:=\eta_{k}(g,h)$ and $\beta^{h,k}_{g}=\nu_{g}(h,k)$, and $g_{*}(h):=ghg^{-1}$, then it's direct calculations show this data satisfies equations \eqref{eq: equation beta}, \eqref{eq: 2-cocycle condition mu}, \eqref{eq: compatibility mu and beta} (see \cite[Equation 21]{MR2552301}) and thus provides a $G$ action on $\text{Vec}(G, \omega)$ via conjugation. It is well known that $\text{Vec}(G, \omega)^{G}\cong \mathcal{Z}(\text{Vec}(G, \omega))$.

Thus according to Corollary \ref{subcategoriespointed}, since the action of $G$ on $G$ is by conjugation, the subcategories of $Z(\text{Vec}(G,\omega))$ are parameterized by triples $(L,H,\eta)$, where $L, H$ are commuting normal subgroups of $G$ and $\eta$ is a G-invariant $(\beta, \mu)$-bicharacter.

The parameterization in \cite{MR2552301} is given by triples $(L,H,B)$, where $L, H$ are commuting normal subgroups of $G$, and $B:L\times H\rightarrow \mathbbm{k}^{\times}$ is a $G$-invariant $\omega$-bicharacter (\cite[Definition 5.4]{MR2552301}). Using \cite[Equation 19]{MR2552301}, we see that $B$ is an $\omega$-bicharacter if and only if $B^{-1}$ is a $(\beta,\mu)$-bicharacter in our sense.

\subsection{Obstruction theory}

In Corollary \ref{subcategoriespointed}, we parameterized the subcategories of the equivariantization of a pointed fusion category in terms of algebraic data, which can be described as solutions to families of equations with a cohomological flavor. It is desirable to have an obstruction theory theory that allows one to determine whether such data exists or not, and to be able to count the solutions if they do exist. In this section, we present some results in this direction.

Let $G$ be a finite group acting on  $\text{Vec}(K, \omega)$ with associated data $\beta_{g}^{k,l}$ and $\mu_{g,h}^k$. Let $L \leq K$ be a $G$-invariant subgroup and $H\leq G$ a normal subgroup such that $H$ acts trivially on $L$. We will analyze the cohomological conditions for the existence of a $G$-equivariant $(\beta,\mu)$-bicharacter $\eta:L \times H\to \ku^\times$.

\subsubsection{Existence of $(\beta,\mu)$ bicharacters.}

Since $H$ acts trivially on $L$, it follows from equation \eqref{eq: equation beta} that  $$(\beta_h)|_{L\times L}\in Z^2(L,\ku^\times),$$ for all $h\in H$. If follows from \eqref{eq: eta is coboundary of beta} that a \emph{first necessary condition} for the existence of a $(\beta,\mu)$ bicharacter is that the cohomology class of $(\beta_h)|_{L\times L}$ is trivial for all $h\in H$. If the cohomology class $(\beta_h)|_{L\times L}$ is trivial for all $h\in H$, there exist $\tau:L\times H\to \ku^\times$ such that \[\tau^{h}_{l_1l_2}=\tau^{h}_{l_1}\tau^{h}_{l_2}\beta_{h}^{l_1,l_2}\]
holds for all $h\in H, l_1,l_2 \in L$.

The function 
\begin{align*}
\tilde{\mu} :H\times H\times L &\to \ku^\times \\  (h_1,h_2,l) &\mapsto \mu_{h_1,h_2}^l \frac{\tau_{l}^{h_1h_2}}{\tau_l^{h_1}\tau_l^{h_2}},%,\\
\end{align*}
defines an element  in $\tilde{\mu} \in Z^2(H,\operatorname{Hom}(L,\ku^\times)).$ In fact, 

\begin{align*}
    \tilde{\mu}^{l_1}_{h_1,h_2}\tilde{\mu}^{l_2}_{h_1,h_2}&=\mu_{h_1,h_2}^{l_1}\mu_{h_1,h_2}^{l_2} \frac{\tau_{l_1}^{h_1h_2}}{\tau_{l_1}^{h_1}\tau_{l_1}^{h_2}} \frac{\tau_{l_2}^{h_1h_2}}{\tau_{l_2}^{h_1}\tau_{l_2}^{h_2}}\\
    &=\mu_{h_1,h_2}^{l_1}\mu_{h_1,h_2}^{l_2}\beta^{l_1,l_2}_{h_1}\beta^{l_1,l_2}_{h_2} \frac{\tau_{l_1}^{h_1 h_2}\tau_{l_2}^{h_{1} h_{2}}}{\tau_{l_1l_2}^{h_1} \tau_{l_1l_2}^{h_2}}\\
    &=\mu^{l_1l_2}_{h_1,h_2}\beta_{h_1 h_2}^{l_1,l_2} \frac{\tau_{l_1}^{h_1 h_2}\tau_{l_2}^{h_{1} h_{2}}}{\tau_{l_1l_2}^{h_1} \tau_{l_1l_2}^{h_2}}\\
    &=\mu^{l_1l_2}_{h_1,h_2}\frac{\tau_{l_1l_2}^{h_1h_2}}{\tau^{h_1}_{l_1l_2}\tau^{h_2}_{l_1l_2}}\\
    &=\tilde{\mu}^{l_1l_2}_{h_1,h_2},
\end{align*}
for all $l_1,l_2\in L, h_1,h_2\in H$. It is easy to see that the cohomology class of $\tilde{\mu} \in Z^2(H,\operatorname{Hom}(L,\ku^\times))$ does not depend of the choice of $\tau$. Hence, it follows from equation \eqref{eq: eta is coboundary of mu} that a \emph{second necessary condition} for the existence of a $(\beta, \mu)$-bicharacter is the triviality of the cohomology of $\tilde{\mu}$. Moreover, the triviality of the cohomology of $\tilde{\mu}$ is a sufficient condition for the existence of a $(\beta, \mu)$-bicharacter. In fact, if $\chi \in \operatorname{Hom}(L,\ku^\times)$ is such that $\tilde{\mu}(h_1,h_2)=\frac{\chi_{h_1h_2}}{\chi_{h_1}\chi_{h_2}}$, then $\eta_h^l=\tau_h^l\chi_h^l$ is a $(\beta,\mu)$-bicharacter.

\subsubsection{Existence of $G$-invariant $(\beta, \mu)$-bicharacters.}
It is easy to see that the set of all $(\beta, \mu)$-characters is a torsor over the abelian group of all bicharacters $H\times L\to \ku^\times$. Let $\eta:H\times L\to \ku^\times$ be a $(\beta, \mu)$-bicharacter. Define the function \begin{align}
\Omega(\eta,\mu):G\times H\times L &\to \ku^\times \\
(g,h,l) &\mapsto \left (\frac{\eta^{l}_{h}}{\eta^{g_{*}(l)}_{ghg^{-1}}} \right)\left( \frac{ \mu^{l}_{ghg^{-1},g}}{\mu^{l}_{g,h}}\right) \notag.
\end{align}

\begin{proposition}
The following conditions are necessary and sufficient for the existence of a $G$-invariant $(\beta, \mu)$-bicharacter.

\begin{enumerate}
    \item $\Omega(\eta,\mu)$ defines an element in $Z^1(G,\operatorname{Hom}(L,\operatorname{Hom}(H,\ku^\times)))$ whose cohomology class does not depend on the choice of $\eta$. In the above group, the action of $G$ on a bicharacter is given by $^g(\chi_h^l)=\chi_{ghg^{-1}}^{g*(l)}$ for all $g\in G, h\in H, l \in L$.
    \item The cohomology class of $\Omega(\eta,\mu)$ vanishes.
    
 \end{enumerate}
\end{proposition}
\begin{proof}
Let $\tau:H\times L \to \ku^\times$ be a $G$-invariant $(\beta, \mu)$-bicharacter. There is a bicharacter $\chi:L\times H\to \ku^\times$ such that $\eta=\tau \chi$. Hence,

\begin{align*}
    \Omega(\eta,\mu)_{g,h,l}=\frac{\chi^{l}_{h}}{\chi^{g_{*}(l)}_{ghg^{-1}}},
\end{align*}and 
 
\begin{align*}
\Omega(\eta,\mu)_{g_1,h,l} \ ^{g_1}\left( \Omega(\eta,\mu)_{g_2,h,l} \right)&=\frac{\chi^{l}_{h}}{\chi^{(g_1)_{*}(l)}_{h}} \ ^{g_1}\left( \frac{\chi^{l}_{h}}{\chi^{(g_2)_{*}(l)}_{h}} \right)\\
&=\frac{\chi^{l}_{h}}{\chi^{(g_1g_2)_{*}(l)}_{h}}\\
&=\Omega(\eta,\mu)_{g_1g_2,h,l},
\end{align*}
so $\Omega(\eta,\mu)\in Z^1(G,\operatorname{Hom}(K,\operatorname{Hom}(H,\ku^\times)))$ and the cohomology does not depend on $\eta$.

It follows by the definition of $\Omega(\eta,\mu)$ that it vanishes if $\eta$ is  $G$-invariant. Conversely, if  $\Omega(\eta,\mu)\in Z^1(G,\operatorname{Hom}(K,\operatorname{Hom}(H,\ku^\times)))$  and there exists a bicharacter $\chi$ such that $\Omega(\eta,\mu)_{g,h,l}=\frac{\chi^{l}_{h}}{\chi^{g_{*}(l)}_{h}}$ then $\tau =\frac{\eta}{\chi}$ is a $G$-invariant
\end{proof}

\section{The lattice of Hopf subalgebras of a Hopf algebra of  Kac-Paljutkin type}\label{application}

\subsection{Hopf algebra of  Kac-Paljutkin type}

Let $G$ be a finite  group acting on a finite group $K$ and $(\beta,\mu)$ a data defining an action of $G$ on $\operatorname{Vec}(K)$.

Let us denote $\ku^G:= {\rm{Maps}}(G,\ku)$, let $\delta_g \in \ku^G$  be the function that assigns $1$ to $g$ and $0$ otherwise, and let $\delta_{g,h}$ be the Dirac's delta associated to the pair $g,h \in G$, namely $\delta_{g,h}$ is $1$ whenever $g= h$ and $0$ otherwise.

The vector space $\ku^G\#_{\mu}^\beta \ku K$ with basis
$\{\delta_g\# x| g\in G, x\in K\}$ is a Hopf algebra
with product
\[(\delta_g\#x)(\delta_h \# y)=
\delta_{g,h}\beta^{x,y}_g \delta_g\#xy,\] coproduct,
\[\Delta(\delta_g\# x)=
\sum_{a,b\in G: ab=g}\mu_{a,b}^x\delta_a\# b\cdot x\ot
\delta_b\#x,\]
counit, unit and antipode
\begin{align*}
\varepsilon(\delta_g\# x)=\delta_{g,e}, &&  1\#e, && S(\delta_g\# x)=\frac{1}{\beta_{(gx)^{-1}}^{gx,g^{-1}}\mu_{g^{-1},g}^x}\delta_{g^{-1}}\# (gx)^{-1},
\end{align*}

for all $g,h\in G$, $x,y\in K$. See \cite{MR1913439} for more details.

\begin{corollary}Let $\ku^G\#_{\mu}^\beta \ku K$ be a semisimple Hopf algebra of Kac-Paljutkin type, then the lattice of Hopf subalgebras is in correspondence with the lattice of triples $(L, H, \eta)$ where

\begin{enumerate}
    \item 
    $L\le K$ is a  $G$-invariant subgroup.
    \item
    $H\le G$ is a normal subgroup.
    \item
    $\eta:H\times L\rightarrow \ku^{\times}$ is a $G$-invariant $(\beta,\mu)$-bicharacter.
\end{enumerate}
\end{corollary}
\begin{proof}
It follows from Corollary \ref{subcategoriespointed} that the triples are in correspondence with fusion subcategories of $\operatorname{Vec}(K)^G$. Now, the fusion category of right $\ku^G\#_\mu^\beta \ku K$-comodules is tensor isomorphic to the fusion category $\operatorname{Vec}(K)^G$, (see \cite[Theorem 7.1]{Galindo-Uribe}) hence  by \cite[Theorem 6]{MR1375826} they correspond to Hopf subalgebras of $\ku^G\#_\mu^\beta \ku K$. 
\end{proof}

\subsection{Concrete example}

Let $G=C_2=\langle \sigma \rangle$ a cyclic group of order two and $K=\Z/n\times \Z/n$. Consider the action of $G$ on $K$ given by $\sigma(a,b)=(b,a)$. Since we can consider only normalized maps $\beta:C_2\times K\times K\to \ku^\times$ $\mu:C_2\times C_2\times K \to \ku^\times$, hence the pair $(\beta,\mu)$ is complete determined by elements $\beta\in Z^2(K,\ku^\times)$ and $\mu \in C^1(K,\ku^\times)$ such that 

\begin{align}
    \delta_K(\mu)=\beta (^\sigma \beta), && \mu(^\sigma\mu)=1.
\end{align}

Up to equivalences, the pairs $(\beta, \mu)$ are classified by $\mathbb{G}_n$ the group $n$-th root of unit, see \cite[Theorem 2.1]{MR1448809}. In fact, given $q\in \mathbb{G}_n$ we can define 
\begin{align}
\beta_q(\vec{a},\vec{b})=q^{a_1b_2}, && \mu_q(\vec{a})=q^{-a_1a_2},    
\end{align}where $\vec{a}=(a_1,a_2)$ and $\vec{b}=(b_1,b_2)$.

We now can divide the possible Hopf subalgrebras of $\ku^{C_2}\#_{\mu_q}^{\beta_q} \ku(\Z/n\times \Z/n)$ in two cases depending if $H\leq C_2$ is trivial or not.

\subsubsection{When $H=C_2$ }\label{H non trivial}

In this case the subgroup $L$ is just a subgroup of $\Delta(\Z/n)=\{(a,a):a\in \Z/n\}$. Let $m$ be a divisor of $n$ and $x=n/m$. Hence $L=\langle (x,x)\rangle$ and \[\beta_q(a(x,x),b(x,x))=q^{x^2ab}.\]

The 1-cochain 
\[\tau_q(l(x,x))=q^{-x^2\frac{l(l+1)}{2}}\] satisfies 
$\delta_K(\tau_q)=\beta_q$ for all $l(x,x)\in L$.

In order to compute the obstruction, let us consider

\begin{align*}
    \tilde{\mu}_q(l(x,x))&=\frac{\mu_q(l(x,x))}{\tau_q(l(x,x))^2}\\
    &=q^{-l^2x^2}q^{x^2l(l+1)}\\
    &=q^{x^2l}.
\end{align*}
Then the obstruction vanishes if and only if there is $\zeta\in \mathbb{G}_m$ such that \[q^{x^2}=\zeta^2.\]
If $q^{x^2}=\zeta^2,$ the associated $(\beta,\mu)$-bicharacter is 

\[\eta_{\zeta}(l(x,x))=\tau_q(l(x,x))\zeta^l=q^{-x^2\frac{l(l+1)}{2}}\zeta^l.\]
Moreover every $\eta_{\zeta}$ is invariant.

\subsubsection{When $H$ is trivial}\label{H trivial}

In general, if $H$ is trival the fusion sub\-ca\-te\-go\-ries associated are given just by pairs $(L,\eta)$ where $L\leq K$ is a $G$-invariant subgroup and $\eta \in \widehat{G}$ is $G$-invariant linear character.

Using Goursat's lemma it is straightforward to see that $C_2$-invariant subgroups of $\Z/n\times \Z/n$ corresponds to triple $(B,N,f)$ where $N\subset B\subset \Z/n$ is a tower of groups and $f:B/N\to B/N$ is an automorphism of order two. The subgroup associated with the triple $(B,N,f)$ is 
\[H_{(B,N,f)}=\{(a,b): f(aN)=bN\}.\]

As a conclusion of the subsections \ref{H non trivial} and \ref{H trivial} we obtain the following result.
\begin{theorem}
The Hopf subalgebras of $\ku^{C_2}\#_{\mu_q}^{\beta_q} \ku(\Z/n\times \Z/n)$ come in two types, and correspond to

\begin{itemize}
    \item[Type 1:]  pairs $(m,\zeta)$ where $m$ is divisor of $n$ and $\zeta$ is a root of unity of order $\frac{n}{m}$ such that $q^{\left( \frac{n}{m}\right )^2}=\zeta^2$.
    \item[Type 2:] four-tuples $(B,N,f,\eta)$ where where $N\subset B\subset \Z/n$ is a tower of groups, $f:B/N\to B/N$ is an automorphism of order two, and $\eta:H_{(B,N,f)}\to \ku^\times$ is a character such that $\eta(a,b)=\eta(b,a)$ for all $(a,b)\in H_{(B,N,f)}.$
\end{itemize}
\end{theorem}
\qed

%\bibliographystyle{alpha}
%\bibliography{biblio}

\end{document}